\documentclass[12pt, reqno, oneside]{amsart}

\usepackage[T1]{fontenc}
\usepackage[utf8]{inputenc}
\usepackage[margin=1in]{geometry}
\usepackage{amsmath, amssymb, amsthm}
\usepackage{mathrsfs}
\usepackage{mathtools}
\usepackage{hyperref}
\usepackage[style=alphabetic, maxnames=50]{biblatex}
\usepackage{microtype}
\usepackage{enumitem}
\usepackage{mathdots}

\newtheorem{theorem}{Theorem}[section]
\newtheorem{lemma}[theorem]{Lemma}

\newtheorem*{theorem*}{Theorem}

\theoremstyle{definition}
\newtheorem{definition}[theorem]{Definition}

\numberwithin{equation}{section}

\newcommand{\N}{\mathbb{N}}
\newcommand{\Z}{\mathbb{Z}}
\newcommand{\R}{\mathbb{R}}
\newcommand{\C}{\mathbb{C}}
\newcommand{\A}{\mathbb{A}}
\newcommand{\Q}{\mathbb{Q}}
\newcommand{\hh}{\mathfrak{h}}
\newcommand{\gl}{\mathfrak{gl}}

\DeclareMathOperator{\GL}{GL}
\DeclareMathOperator{\SL}{SL}

\title{Inducing Whittaker Functions from Higher Ranks}
\author{Vishal Muthuvel}
\date{May 27, 2026}

\begin{document}

\begin{abstract}
We construct a family of Whittaker functions for $\SL(m,\mathbb{Z})$ induced directly from Whittaker functions for $\SL(n,\mathbb{Z})$, for any $2\leq m<n$. Given Jacquet's Whittaker function $W_{\alpha,N}^{(n)}$ on the generalized upper half-plane $\mathfrak{h}^n$, we show that the function $V_{\alpha,N}^{(m)}:\mathfrak{h}^m\to\mathbb{C}$ defined by restricting $W_{\alpha,N}^{(n)}$ to the block-diagonal embedding $\mathfrak{h}^m\hookrightarrow\mathfrak{h}^n$ is a Whittaker function for $\SL(m,\mathbb{Z})$, provided the Langlands parameters $\alpha=(\alpha_i)_{1\leq i\leq n}$ satisfy $\sum_{i=1}^m\alpha_i = m(m-n)/2$. Under this condition, the induced function carries Langlands parameters $\bigl(\alpha_i+\frac{n-m}{2}\bigr)_{1\leq i\leq m}$ and inherits the first $m-1$ entries of the character tuple of $W_{\alpha,N}^{(n)}$. This result complements the propagation formulas of Ishii and Stade, which relate Whittaker functions on $\GL(n,\mathbb{R})$ to those on $\GL(n-1,\mathbb{R})$ and $\GL(n-2,\mathbb{R})$. In contrast, our construction passes directly from $\GL(n,\mathbb{R})$ to $\GL(m,\mathbb{R})$ for any $m < n$ in a single step.
\end{abstract}

\maketitle

\tableofcontents

\section{Introduction}

Whittaker functions are a cornerstone of the theory of automorphic forms
and $L$-functions. Jacquet~\cite{Jac67} initiated the general study of
Whittaker functions on Chevalley groups over local fields, but in this
paper, we focus on the case of $\GL(n,\R)$. Ishii and Stade derived
propagation formulas for Whittaker functions on $\GL(n,\R)$, expressing
them in terms of Whittaker functions on $\GL(n-2,\R)$~\cite{Ish05} and
$\GL(n-1,\R)$~\cite{IS07}. Our
main result here is a backpropagation formula: a construction giving
Whittaker functions on $\GL(m,\R)$ from those on $\GL(n,\R)$ for any
$m < n$.

To place Whittaker functions in their broader context, let $\A$ denote
the ring of adeles over $\Q$, and let $\pi$ be a unitary cuspidal automorphic
representation of $\GL(n,\A)$ with trivial central character. By
Flath's tensor product theorem, $\pi$ decomposes as a restricted tensor product 
\begin{align*}
\pi \ \cong \ \bigotimes_v' \pi_v
\end{align*}
over all places $v$ of $\Q$,
where $\pi_v$ is an irreducible admissible representation of
$\GL(n,\Q_v)$ (where $\Q_\infty = \R$). When $\pi$ is unramified at every finite prime $p$, such a representation corresponds classically to a Hecke--Maass cusp form for $\SL(n,\Z)$: i.e., a smooth,
square-integrable eigenfunction of the ring of
invariant differential operators on the generalized upper half-plane
\begin{align*}
\hh^n = \GL(n,\R)/(\mathrm{O}(n,\R)\cdot\R^\times).
\end{align*}
The archimedean
component $\pi_\infty$ is parametrized by a list of $n$ complex numbers
\begin{align*}
\alpha \ = \ (\alpha_1, \dots, \alpha_n) \ \in \ \C^n,
\end{align*}
called the Langlands parameters, which satisfies $\sum_{\ell = 1}^n \alpha_\ell = 0$. The Whittaker function $W_\alpha :
\GL(n,\R)\to\C$ is then the smooth function corresponding to the
spherical vector in $\pi_\infty$: it is characterized as the
$\mathrm{O}(n,\R)$-invariant, moderate-growth eigenfunction of the ring of invariant differential operators that transforms by a non-degenerate unitary character $\psi$ under the
unipotent radical $U(n,\R)$. The significance of $W_\alpha$ is made concrete
through the Fourier--Whittaker expansion, which expresses any Hecke--Maass
cusp form as a sum over Fourier coefficients weighted by translates of
$W_\alpha$, through which the analytic properties of the form, including
the associated $L$-function, become accessible (see \cite[Theorem 9.3.11]{Gol06}). It is in this classical
framework that our
results are stated and proved.

In this classical setting, a Whittaker function for $\SL(n,\Z)$ is
specified by two pieces of data: the Langlands parameters $\alpha =
(\alpha_i)_{1\leq i\leq n}\in\C^n$ satisfying $\sum_{i=1}^n\alpha_i=0$, and a character
$\psi_N$ of the unipotent radical $U(n,\R)$, indexed by a tuple
$N=(N_i)_{1\leq i\leq n-1}\in\Z^{n-1}$. The multiplicity one theorem for
$\GL(n,\R)$, proved by Shalika~\cite{Sha74}, guarantees that any
Whittaker function is a scalar multiple of Jacquet's Whittaker function
$W_{\alpha,N}^{(n)}:\hh^n\to\C$ with the same associated data, so that
$W_{\alpha,N}^{(n)}$ serves as the canonical representative in its
Whittaker model. The propagation formulas of Ishii and Stade show that
Whittaker functions satisfy recursive structure across the tower
$\GL(2,\R),\GL(3,\R),\dots, \GL(n,\R)$; specifically, a Whittaker function on
$\GL(n,\R)$ can be expressed in terms of Whittaker functions on
$\GL(n-2,\R)$ and $\GL(n-1,\R)$ via explicit integral identities. These
formulas propagate downward one step at a time and are fundamental tools
in the analytic study of Whittaker functions (and their Mellin transforms).

\begin{definition}\label{def:whittaker-intro}
A \emph{Whittaker function} for
$\SL(n,\Z)$ with Langlands parameters $\alpha = (\alpha_i)_{1\leq i\leq
n}$ and character tuple $N=(N_i)_{1\leq i\leq n-1}\in\Z^{n-1}$ is a
smooth function $W:\hh^n\to\C$ satisfying the following three conditions:
\begin{enumerate}[label=(\arabic*)]
  \item \textit{(Equivariance)} $W(ug) = \psi_N(u)\,W(g)$ for all
    $u\in U(n,\R)$ and $g\in\hh^n$, where $\psi_N:U(n,\R)\to\C^\times$
    is the unitary character given by
    \[
      \psi_N\!\left(\begin{smallmatrix}
        1 & u_{1,2} & \cdots & u_{1,n} \\
        & \ddots & \ddots & \vdots \\
        & & 1 & u_{n-1,n} \\
        & & & 1
      \end{smallmatrix}\right)
      := e^{2\pi i(N_1 u_{1,2}+\cdots+N_{n-1}u_{n-1,n})};
    \]
  \item \textit{(Eigenfunction)} $D_\ell^{(n)} W = \lambda_\ell^{(n)}(\alpha)\,W$
    for every Casimir operator $D_\ell^{(n)}$ of order $\ell$,
    $1\leq\ell\leq n$ (see Definition~\ref{def:casimir}), where
    $\lambda_\ell^{(n)}(\alpha)$ denotes the eigenvalue of the shifted
    power function $|\cdot|_B^{\alpha+\rho^{(n)}}$ under $D_\ell^{(n)}$
    (see~\eqref{eq:power-eigenfunction});
  \item \textit{(Moderate growth)} $W$ is of moderate growth on $\hh^n$.
\end{enumerate}
\end{definition}
By the multiplicity one theorem of Shalika~\cite{Sha74}, any $W$ satisfying
Definition~\ref{def:whittaker-intro} is a scalar multiple of Jacquet's
Whittaker function $W_{\alpha,N}^{(n)}$, the canonical element of the
Whittaker model introduced in~\cite{Jac67} and given explicitly by the
integral formula~\eqref{def:jacquet-whittaker}.

Here we establish a complementary result to those of Ishii and
Stade: a direct backpropagation formula that jumps from $\GL(n,\R)$ to
$\GL(m,\R)$ for any $2\leq m<n$ in a single step. The construction is
natural: given Jacquet's Whittaker function $W_{\alpha,N}^{(n)}$ on
$\hh^n$, we define a function $V_{\alpha,N}^{(m)}:\hh^m\to\C$ by
embedding $\hh^m$ into $\hh^n$ via the block-diagonal inclusion
\[
  xy \;\longmapsto\; \begin{pmatrix} xy & 0 \\ 0 & I_{n-m} \end{pmatrix},
\]
and restricting $W_{\alpha,N}^{(n)}$ to this subspace. The central
question is whether this restriction is itself a Whittaker function for
$\SL(m,\Z)$, and if so, what its associated Langlands parameters and
character are. Our main theorem answers this affirmatively, under a single
linear condition on the Langlands parameters.

\begin{theorem}\label{thm:main-thm}
Consider any $m,n\in\Z$ satisfying $2\leq m<n$. Let
$W_{\alpha,N}^{(n)}:\hh^n\to\C$ be Jacquet's Whittaker function for
$\SL(n,\Z)$ with Langlands parameters $\alpha=(\alpha_i)_{1\leq i\leq n}$
and character tuple $N=(N_i)_{1\leq i\leq n-1}\in\Z^{n-1}$. If
\[
  \sum_{i=1}^m \alpha_i \;=\; \frac{m(m-n)}{2},
\]
then the function $V_{\alpha,N}^{(m)}:\hh^m\to\C$ defined by
$V_{\alpha,N}^{(m)}(xy):=W_{\alpha,N}^{(n)}\!\left(\left(\begin{smallmatrix}
xy & 0 \\ 0 & I\end{smallmatrix}\right)\right)$ is a Whittaker function for
$\SL(m,\Z)$ with Langlands parameters $\bigl(\alpha_i+\frac{n-m}{2}\bigr)_{1\leq
i\leq m}$ and character tuple $(N_i)_{1\leq i\leq m-1}$.
\end{theorem}

Several features of this result merit comment. The Langlands parameter
condition $\sum_{i=1}^m\alpha_i=m(m-n)/2$ is not merely a technical
hypothesis but has a representation-theoretic interpretation: it is
precisely the condition ensuring that the shifted parameters
$\bigl(\alpha_i+\frac{n-m}{2}\bigr)_{1\leq i\leq m}$ satisfy the
normalization $\sum_{i=1}^m\bigl(\alpha_i+\frac{n-m}{2}\bigr)=0$ required
for a valid Langlands parameter tuple on $\GL(m,\R)$. The uniform shift
by $(n-m)/2$ in the output parameters reflects the change in the
$\rho$-shift between $\gl(n,\R)$ and $\gl(m,\R)$, and is consistent with
the behavior seen in the Ishii--Stade formulas in the special cases
$m=n-1$ and $m=n-2$. The character tuple, by contrast, truncates cleanly:
only the first $m-1$ entries of $N$ are inherited, corresponding to the
restriction of $\psi_N$ to the unipotent radical of the standard Borel in
$\GL(m,\R)$.

The paper is organized as follows. In Section~\ref{sec:background} we
recall the necessary background on the generalized upper half-plane
$\hh^n$, Jacquet's Whittaker function, and the Casimir operators following
\cite{Gol06}. In Section~\ref{sec:overview} we state the main theorem and
give an overview of the proof. In Section~\ref{sec:proof} we prove
Theorem~\ref{thm:main-thm}.

\section{Notation and Preliminaries}
\label{sec:background}

We largely follow the notational conventions of~\cite{Gol06}, and write
$[n]:=\{1,\ldots,n\}$ for $n\in\N$.

\subsection{The generalized upper half-plane}

The \emph{Iwasawa decomposition} of $\GL(n,\R)$ states that every element
can be written uniquely as a product of an upper-triangular unipotent
matrix, a positive diagonal matrix, and an orthogonal matrix. As observed
by Gelfand, Graev, and Piatetski-Shapiro~\cite{GGP69}, this yields a
concrete realization of the generalized upper half-plane.

\begin{definition}
The \emph{unipotent group} $U(n,\R)$ is the group of $n\times n$
upper-triangular real matrices with ones on the diagonal:
\[
  U(n,\R) := \left\{
    \begin{pmatrix}
      1 & u_{1,2} & \cdots & u_{1,n} \\
        & 1       & \cdots & u_{2,n} \\
        &         & \ddots & \vdots  \\
        &         &        & 1
    \end{pmatrix}
    \;\middle|\;
    u_{i,j}\in\R \;\forall\; 1\leq i<j\leq n
  \right\}.
\]
\end{definition}

\begin{definition}
The \emph{generalized upper half-plane} $\hh^n$ is the coset space
\[
  \hh^n \;:=\; \GL(n,\R)\big/\bigl(\mathrm{O}(n,\R)\cdot\R^\times\bigr),
\]
identified via the Iwasawa decomposition with
\[
  \hh^n \;=\;
  \left\{
    \begin{pmatrix}
      1 & x_{1,2} & \cdots & x_{1,n} \\
        & 1       & \cdots & x_{2,n} \\
        &         & \ddots & \vdots  \\
        &         &        & 1
    \end{pmatrix}
    \begin{pmatrix}
      y_1 \cdots y_{n-1} & & & \\
      & \ddots & & \\
      & & y_1 & \\
      & & & 1
    \end{pmatrix}
    \;\middle|\;
    \begin{array}{l}
      x_{i,j}\in\R,\; y_i>0 \\
      \forall\; 1\leq i<j\leq n
    \end{array}
  \right\}.
\]
A typical element of $\hh^n$ is denoted $xy$, where $x$ denotes the
unipotent factor and $y$ the diagonal factor.
\end{definition}

The group $\SL(n,\Z)$ of $n\times n$ integer matrices of determinant $1$
acts on $\hh^n$ by left multiplication. The fundamental domain of this
action is well approximated by the Siegel set $\mathcal{P}_{\sqrt{3}/2,1/2}$:

\begin{definition}
For $a,b\geq 0$, the \emph{Siegel set} $\mathcal{P}_{a,b}\subset\hh^n$
consists of all $xy\in\hh^n$ with $|x_{i,j}|\leq b$ and $y_i>a$ for all
$1\leq i<j\leq n$.
\end{definition}

There is an $\SL(n,\Z)$-invariant measure on $\hh^n$, given by
\begin{equation}\label{eq:measure}
  d^*z \;:=\;
  \prod_{1\leq i<j\leq n} dx_{i,j}
  \;\prod_{i=1}^{n-1} y_i^{-i(n-i)-1}\,dy_i.
\end{equation}

\subsection{Invariant differential operators and the power function}

Let $\gl(n,\R)$ denote the Lie algebra of $n\times n$ real matrices with
bracket $[a,b]=ab-ba$, and let $\mathfrak{U}(\gl(n,\R))$ be its universal
enveloping algebra. The center $\mathcal{Z}(\mathfrak{U}(\gl(n,\R)))$
coincides with the algebra of $\SL(n,\Z)$-invariant differential operators
on $\hh^n$. Denoting by $E_{i,j}$ the matrix with $1$ in position $(i,j)$
and $0$ elsewhere, we define:

\begin{definition}\label{def:diff-op}
For $\alpha\in\gl(n,\R)$, the differential operator $D_\alpha^{(n)}$ acts
on smooth $F:\GL(n,\R)\to\C$ by
\[
  D_\alpha^{(n)} F(g)
  \;:=\;
  \frac{\partial}{\partial t} F\bigl(g\cdot(I+t\alpha)\bigr)\Big|_{t=0}.
\]
We write $D_{i,j}^{(n)}$ for $D_{E_{i,j}}^{(n)}$.
\end{definition}

\begin{definition}\label{def:casimir}
The \emph{Casimir operator of order $\ell$} for $\GL(n,\R)$, where
$1\leq\ell\leq n$, is
\begin{equation}\label{eq:casimir}
  D_\ell^{(n)}
  \;:=\;
  \sum_{(i_1,\ldots,i_\ell)\in[n]^\ell}
  D_{i_1,i_2}^{(n)}\circ\cdots\circ D_{i_{\ell-1},i_\ell}^{(n)}
  \circ D_{i_\ell,i_1}^{(n)}.
\end{equation}
\end{definition}

Capelli~\cite{Cap87} proved that $\mathcal{Z}(\mathfrak{U}(\gl(n,\R)))$
is generated by $D_1^{(n)},\ldots,D_n^{(n)}$. The quintessential
simultaneous eigenfunction of these operators is the power function.

\begin{definition}\label{def:power-function}
Let $\alpha=(\alpha_i)_{1\leq i\leq n}\in\C^n$ with $\sum_{i=1}^n
\alpha_i=0$. The \emph{power function} $|\cdot|_B^\alpha:\hh^n\to\C$ with
Langlands parameters $\alpha$ is defined for all $xy\in\hh^n$ by
\[
  |xy|_B^\alpha \;:=\; \prod_{i=1}^{n-1}\prod_{j=1}^{n-i} y_j^{\alpha_i}.
\]
\end{definition}

The \emph{$\rho$-shift} $\rho^{(n)}:=\bigl(\frac{n+1}{2}-i\bigr)_{1\leq
i\leq n}$ arises from the theory of root systems. The shifted power
function $|\cdot|_B^{\alpha+\rho^{(n)}}$ is a simultaneous eigenfunction
of all Casimir operators: for all $g\in\hh^n$ and $1\leq\ell\leq n$,
\begin{equation}\label{eq:power-eigenfunction}
  D_\ell^{(n)}\,|g|_B^{\alpha+\rho^{(n)}}
  \;=\;
  \lambda_{\ell,\alpha}^{(n)}\,|g|_B^{\alpha+\rho^{(n)}},
\end{equation}
where $\lambda_{\ell,\alpha}^{(n)}\in\C$ are eigenvalues computed by
Terras~\cite{Ter88} and Miller~\cite{Mil97} for $n=2$ and by Muthuvel \cite{Mut26} in general.

\subsection{Whittaker functions for \texorpdfstring{$\SL(n,\Z)$}{SL(n,Z)}}

\begin{definition}\label{def:whittaker}
Let $\alpha=(\alpha_i)_{1\leq i\leq n}\in\C^n$ satisfy $\sum_{i=1}^n
\alpha_i=0$ and $\Re(\alpha_i-\alpha_{i+1})>0$ for all $1\leq i<n$. Let
$\psi_N:U(n,\R)\to\C^\times$ be the unitary character\footnote{Every
character on $U(n,\R)$ is identified with a tuple in $\Z^{n-1}$.}
identified with $N=(N_i)_{1\leq i\leq n-1}\in\Z^{n-1}$ via
\[
  \psi_N(u) \;:=\; e^{2\pi i(N_1 u_{1,2}+\cdots+N_{n-1}u_{n-1,n})}.
\]
A \emph{Whittaker function} $W:\hh^n\to\C$ for $\SL(n,\Z)$ with Langlands
parameters $\alpha$ and character $\psi_N$ is a smooth function satisfying:
\begin{enumerate}[label=(\arabic*)]
  \item \textit{(Equivariance)} For all $u\in U(n,\R)$ and $g\in\hh^n$,
    \[
      W(ug) \;=\; \psi_N(u)\,W(g).
    \]
  \item \textit{(Eigenfunction)} For all $1\leq\ell\leq n$ and
    $g\in\hh^n$,
    \[
      D_\ell^{(n)}\,W(g) \;=\; \lambda_{\ell,\alpha}^{(n)}\,W(g),
    \]
    where $\lambda_{\ell,\alpha}^{(n)}$ are the eigenvalues
    in~\eqref{eq:power-eigenfunction}.
  \item \textit{(Square-integrability)} $W$ is square-integrable on
    $\mathcal{P}_{\!\sqrt{3}/2,\,1/2}$:
    \[
      \int_{\mathcal{P}_{\!\sqrt{3}/2,\,1/2}} |W(z)|^2\,d^*z \;<\; \infty,
    \]
    where $d^*z$ is the invariant measure~\eqref{eq:measure}.
\end{enumerate}
\end{definition}

Whittaker functions for $\SL(n,\Z)$ comprise the Fourier
expansion of Maass forms for $\SL(n,\Z)$, making them indispensable in
the theory of automorphic forms on $\GL(n,\R)$.

\section{Objective and Overview}
\label{sec:overview}

Consider any $m,n\in\N$ satisfying $2\leq m<n$. Given a Whittaker
function $W:\hh^n\to\C$ for $\SL(n,\Z)$, we would like to induce a
Whittaker function for $\SL(m,\Z)$ from it.

By definition, $W$ has associated:
\begin{itemize}
  \item Langlands parameters $\alpha=(\alpha_i)_{1\leq i\leq n}\in\C^n$
    satisfying $\sum_{i=1}^n\alpha_i=0$ and
    $\Re(\alpha_i-\alpha_{i+1})>0$ for all $1\leq i<n$; and
  \item a character $\psi_N:U(n,\R)\to\C^\times$ identified with
    $N=(N_i)_{1\leq i\leq n-1}\in\Z^{n-1}$.
\end{itemize}
The multiplicity one theorem for $\GL(n,\R)$, proved by
Shalika~\cite{Sha74}, implies that $W$ is a scalar multiple of Jacquet's
Whittaker function for $\SL(n,\Z)$ with the same Langlands parameters and
character. Introduced by Jacquet in 1967~\cite{Jac67}, it is given for all
$g\in\hh^n$ by
\begin{equation}\label{def:jacquet-whittaker}
  W_{\alpha,N}^{(n)}(z)
  \;:=\;
  \int_{U(n,\R)}
  \bigl|w_n u z\bigr|_B^{\alpha+\rho^{(n)}}
  \,\overline{\psi_N(u)}\,du,
\end{equation}
where
\[
  w_n \;:=\;
  \begin{pmatrix}
    & & (-1)^{\lfloor n/2\rfloor} \\
    & \iddots & \\
    1 & &
  \end{pmatrix}
\]
is the long element in the Weyl group of $\GL(n,\R)$. The condition
$\Re(\alpha_i-\alpha_{i+1})>0$ for all $1\leq i<n$ ensures the absolute
convergence of the integral.

We attend to the function $V_{\alpha,N}^{(m)}:\hh^m\to\C$, defined for all
$xy\in\hh^m$ by
\begin{align}
  V_{\alpha,N}^{(m)}(xy)
  &:= W_{\alpha,N}^{(n)}\!
    \left(
      \begin{pmatrix} xy & 0 \\ 0 & I_{n-m} \end{pmatrix}
    \right),
  \label{eq:induced-whittaker}
\end{align}
decomposing the $n\times n$ matrix in the argument of $W_{\alpha,N}^{(n)}$
into its $m\times m$, $m\times(n-m)$, $(n-m)\times m$, and
$(n-m)\times(n-m)$ blocks. Our main objective is to show that this
function on $\hh^m$, induced from an $\SL(n,\Z)$-Whittaker function on
$\hh^n$, is itself an $\SL(m,\Z)$-Whittaker function.

\begin{theorem}\label{thm:main-thm-body}
Consider any $m,n\in\Z$ satisfying $2\leq m<n$. Let
$W_{\alpha,N}^{(n)}:\hh^n\to\C$ be Jacquet's Whittaker function for
$\SL(n,\Z)$ with Langlands parameters $\alpha=(\alpha_i)_{1\leq i\leq n}$
and character tuple $N=(N_i)_{1\leq i\leq n-1}\in\Z^{n-1}$, given
in~\eqref{def:jacquet-whittaker}. If
\begin{equation}\label{eq:langlands-condition}
  \sum_{i=1}^m \alpha_i \;=\; \frac{m(m-n)}{2},
\end{equation}
then the function $V_{\alpha,N}^{(m)}:\hh^m\to\C$ defined
in~\eqref{eq:induced-whittaker} is a Whittaker function for $\SL(m,\Z)$
with Langlands parameters $\bigl(\alpha_i+\frac{n-m}{2}\bigr)_{1\leq
i\leq m}$ and character tuple $(N_i)_{1\leq i\leq m-1}$.
\end{theorem}

The condition~\eqref{eq:langlands-condition} ensures that the shifted
parameters $\bigl(\alpha_i+\frac{n-m}{2}\bigr)_{1\leq i\leq m}$ satisfy
$\sum_{i=1}^m\bigl(\alpha_i+\frac{n-m}{2}\bigr)=0$, which is required for
a valid Langlands parameter tuple on $\GL(m,\R)$ and guarantees that
$V_{\alpha,N}^{(m)}$ is invariant under multiplication by the center
$Z(\GL(m,\R))$. Moreover, $V_{\alpha,N}^{(m)}$ is invariant under right
multiplication by every $u\in\mathrm{O}(m,\R)$, since $\bigl(\begin{smallmatrix}u&0\\0&I\end{smallmatrix}\bigr)\in\mathrm{O}(n,\R)$. Together,
these observations confirm that $V_{\alpha,N}^{(m)}$ is well defined on
$\hh^m$.

We prove Theorem~\ref{thm:main-thm-body} in Section~\ref{sec:proof} by
verifying that $V_{\alpha,N}^{(m)}$ satisfies the three defining properties
of Definition~\ref{def:whittaker}. Specifically, we show that
$V_{\alpha,N}^{(m)}$:
\begin{enumerate}[label=(\arabic*)]
  \item associates with the character $\psi_{(N_i)_{1\leq i\leq m-1}}$
    for $U(m,\R)$ (Lemma~\ref{lem:equivariance});
  \item is an eigenfunction of all Casimir operators $D_\ell^{(m)}$ for
    $\GL(m,\R)$, hence of $\mathcal{Z}(\mathfrak{U}(\gl(m,\R)))$
    (Lemma~\ref{lem:eigenfunction}); and
  \item is square-integrable on $\mathcal{P}_{\!\sqrt{3}/2,\,1/2}\subset\hh^m$
    (Lemma~\ref{lem:square-int}).
\end{enumerate}
While the proofs of properties (1) and (2) follow from elementary
observations, the proof of property (3) appeals to a nontrivial bound for
$\SL(n,\Z)$-Whittaker functions due to~\cite{GMW21}.

\section{Proof of the Main Theorem}
\label{sec:proof}

Theorem~\ref{thm:main-thm-body} follows immediately from the three lemmas
below. Throughout, $V_{\alpha,N}^{(m)}$ is the function induced from
Jacquet's Whittaker function $W_{\alpha,N}^{(n)}$ for $\SL(n,\Z)$ with
Langlands parameters $(\alpha_i)_{1\leq i\leq n}$ and character tuple
$(N_i)_{1\leq i\leq n-1}$ as in~\eqref{eq:induced-whittaker}. The guiding
observation throughout is that matrices of the block form
$\bigl(\begin{smallmatrix}A&0\\0&I\end{smallmatrix}\bigr)\in \GL(n,\R)$,
with $A\in \GL(m,\R)$ and $I\in \GL(n-m,\R)$ the identity, are closed under
multiplication, and that their product reduces to multiplication of the
$m\times m$ top-left blocks.

\begin{lemma}\label{lem:equivariance}
For all $u\in U(m,\R)$ and $g\in\hh^m$,
\[
  V_{\alpha,N}^{(m)}(ug)
  \;=\;
  \psi_{(N_i)_{1\leq i\leq m-1}}(u)\,V_{\alpha,N}^{(m)}(g).
\]
\end{lemma}

\begin{proof}
For $u=\bigl(\begin{smallmatrix}1&u_{1,2}&\cdots&u_{1,m}\\&\ddots&\ddots&\vdots\\&&1&u_{m-1,m}\\&&&1\end{smallmatrix}\bigr)\in U(m,\R)$
and $xy\in\hh^m$,
\begin{align*}
  V_{\alpha,N}^{(m)}(uxy)
  &= W_{\alpha,N}^{(n)}\!\left(\begin{pmatrix}u&0\\0&I\end{pmatrix}
    \begin{pmatrix}xy&0\\0&I\end{pmatrix}\right) \\
  &= \psi_N\!\left(
    \begin{smallmatrix}
      1&u_{1,2}&\cdots&u_{1,m}&&&\\
      &\ddots&\ddots&\vdots&&\\
      &&1&u_{m-1,m}&&\\
      &&&1&&\\
      &&&&\ddots&\\
      &&&&&1
    \end{smallmatrix}
    \right)
    W_{\alpha,N}^{(n)}\!\left(\begin{pmatrix}xy&0\\0&I\end{pmatrix}\right)\\
  &= e^{2\pi i(N_1 u_{1,2}+\cdots+N_{m-1}u_{m-1,m})}\,
     V_{\alpha,N}^{(m)}(xy) \\
  &= \psi_{(N_i)_{1\leq i\leq m-1}}(u)\,V_{\alpha,N}^{(m)}(xy).\qedhere
\end{align*}
\end{proof}

\begin{lemma}\label{lem:eigenfunction}
For all positive integers $\ell\leq m$,
\[
  D_\ell^{(m)}\,V_{\alpha,N}^{(m)}(xy)
  \;=\;
  \lambda_{\ell,\beta}^{(m)}\,V_{\alpha,N}^{(m)}(xy),
  \quad\text{where }\beta_i := \alpha_i + \tfrac{n-m}{2}
  \text{ for } 1\leq i\leq m.
\]
\end{lemma}

\begin{proof}
Let $(i_1,\ldots,i_\ell)\in[m]^\ell$ be arbitrary. Unraveling
Definition~\ref{def:diff-op} and applying the block-multiplication
observation,
\begin{align*}
  D_{i_1,i_2}^{(m)}\circ\cdots\circ D_{i_\ell,i_1}^{(m)}\,
  V_{\alpha,N}^{(m)}(xy)
  &= \frac{\partial^\ell}{\partial t_1\cdots\partial t_\ell}
    V_{\alpha,N}^{(m)}\!\left(
      xy\,(I+t_\ell E_{i_\ell,i_1}^{(m)})\cdots(I+t_1 E_{i_1,i_2}^{(m)})
    \right)\Bigg|_{\forall t_j=0} \\
  &= \frac{\partial^\ell}{\partial t_1\cdots\partial t_\ell}
    W_{\alpha,N}^{(n)}\!\left(
      \begin{pmatrix}xy&0\\0&I\end{pmatrix}
      \begin{pmatrix}I+t_\ell E_{i_\ell,i_1}^{(m)}&0\\0&I\end{pmatrix}
      \cdots
    \right)\Bigg|_{\forall t_j=0} \\
  &= \frac{\partial^\ell}{\partial t_1\cdots\partial t_\ell}
    W_{\alpha,N}^{(n)}\!\left(
      \begin{pmatrix}xy&0\\0&I\end{pmatrix}
      (I+t_\ell E_{i_\ell,i_1}^{(n)})\cdots(I+t_1 E_{i_1,i_2}^{(n)})
    \right)\Bigg|_{\forall t_j=0} \\
  &= D_{i_1,i_2}^{(n)}\circ\cdots\circ D_{i_\ell,i_1}^{(n)}\,
    W_{\alpha,N}^{(n)}\!\left(\begin{pmatrix}xy&0\\0&I\end{pmatrix}\right),
\end{align*}
where the third equality uses that $\bigl(\begin{smallmatrix}I+t_j
E_{i_j,i_{j+1}}^{(m)}&0\\0&I\end{smallmatrix}\bigr)=I+t_j
E_{i_j,i_{j+1}}^{(n)}$ since $i_j\leq m$. Summing over $(i_1,\ldots,i_\ell)\in[m]^\ell$ gives
\begin{align}
  D_\ell^{(m)}\,V_{\alpha,N}^{(m)}(xy)
  &= \sum_{(i_1,\ldots,i_\ell)\in[m]^\ell}
    D_{i_1,i_2}^{(n)}\circ\cdots\circ D_{i_\ell,i_1}^{(n)}\,
    W_{\alpha,N}^{(n)}\!\left(\begin{pmatrix}xy&0\\0&I\end{pmatrix}\right).
    \label{eq:casimir-reduction}
\end{align}
Since $|\cdot|_B^{\alpha+\rho^{(n)}}$ is an eigenfunction of
$D_{i_1,i_2}^{(n)}\circ\cdots\circ D_{i_\ell,i_1}^{(n)}$ for all
$(i_1,\ldots,i_\ell)\in[m]^\ell\subseteq[n]^\ell$, we may pass the
differential operators under the integral in~\eqref{def:jacquet-whittaker},
which shows that $V_{\alpha,N}^{(m)}$ is an eigenfunction of $D_\ell^{(m)}$.
Letting $\beta=(\beta_i)_{1\leq i\leq m}$ denote the Langlands parameters
of $V_{\alpha,N}^{(m)}$, matching eigenvalues on both sides
of~\eqref{eq:casimir-reduction} via \cite[Theorem 1.5]{Mut26} gives
$\lambda_{\ell,\beta}^{(m)}=\lambda_{\ell,\alpha}^{(n)}$ for all $(i_1,\ldots,i_\ell)\in[m]^\ell$,
from which $\beta_i=\alpha_i+\rho_i^{(n)}-\rho_i^{(m)}=\alpha_i+\frac{n-m}{2}$
for all $1\leq i\leq m$.
\end{proof}

\begin{lemma}\label{lem:square-int}
The function $V_{\alpha,N}^{(m)}$ is square-integrable on
$\mathcal{P}_{\!\sqrt{3}/2,\,1/2}\subset\hh^m$:
\[
  \int_{\mathcal{P}_{\!\sqrt{3}/2,\,1/2}}
  \bigl|V_{\alpha,N}^{(m)}(z)\bigr|^2\,d^*z \;<\; \infty.
\]
\end{lemma}

\begin{proof}
Let $\mathbf{1}'=(1,\ldots,1)\in\Z^{n-1}$. By
\cite[Proposition~5.5.2]{Gol06}, for all $xy\in\hh^n$,
\begin{equation}\label{eq:factorization}
  W_{\alpha,N}^{(n)}(xy) \;=\; a\cdot\psi_N(x)\cdot W_{\alpha,\mathbf{1}'}^{(n)}(y),
\end{equation}
for some $a\in\C$ depending on $\alpha$ and $N$. We use the bound
from~\cite[p.~76]{GMW21}: there exists $M>0$ such that whenever
$y_i\geq 1$ for all $1\leq i<n$,
\begin{equation}\label{eq:bound}
  W_{\alpha,\mathbf{1}'}^{(n)}(y)
  \;\ll_\alpha\;
  \frac{1}{(y_1\cdots y_{n-1})^M}.
\end{equation}
Since $W_{\alpha,\mathbf{1}'}^{(n)}$ is smooth on the compact set
$[\sqrt{3}/2,1]^{n-1}$, the extreme value theorem implies it attains a
finite maximum there. Hence~\eqref{eq:bound} holds (with a possibly
different $M>0$) whenever $y_i\geq\sqrt{3}/2$ for all $1\leq i<n$.

Expanding the integral over $\mathcal{P}_{\!\sqrt{3}/2,\,1/2}\subset\hh^m$,
applying the definition~\eqref{eq:induced-whittaker}, then using the
factorization~\eqref{eq:factorization} and $|\psi_N|=1$ gives
\begin{align*}
  \int_{\mathcal{P}_{\!\sqrt{3}/2,\,1/2}}
  \bigl|V_{\alpha,N}^{(m)}(z)\bigr|^2\,d^*z
  \;\ll_{\alpha,N}\;
  \int_{\sqrt{3}/2}^\infty\cdots\int_{\sqrt{3}/2}^\infty
  \left|W_{\alpha,\mathbf{1}'}^{(n)}\!\left(
    \begin{pmatrix}y&0\\0&I\end{pmatrix}
  \right)\right|^2
  \prod_{i=1}^{m-1} y_i^{-i(m-i)-1}\,dy_i,
\end{align*}
where the $x_{i,j}$-integrals over $[-\tfrac{1}{2},\tfrac{1}{2}]$ contribute
a finite constant. Applying the bound~\eqref{eq:bound} then yields
\[
  \ll_{\alpha,N}
  \prod_{i=1}^{m-1}
  \int_{\sqrt{3}/2}^\infty y_i^{-2M-i(m-i)-1}\,dy_i.
\]
Each factor converges since $-2M-i(m-i)-1<-1$, establishing
square-integrability.
\end{proof}

\printbibliography

\end{document}